\documentclass[12pt]{article}

\usepackage{latexsym,amsfonts,amsmath,epsfig,tabularx,amsthm,dsfont}

\usepackage{color}
\definecolor{darkred}{RGB}{139,0,0}
\definecolor{darkgreen}{RGB}{0,100,0}
\definecolor{darkmagenta}{RGB}{139,0,139}
\definecolor{darkpurple}{RGB}{110,0,180}
\definecolor{darkblue}{RGB}{40,0,200}
\definecolor{darkorange}{RGB}{255,140,0}

\theoremstyle{definition}

\newtheorem{thm}{Theorem} %  [section]
\newtheorem{algorithm}{Algorithm}
\newtheorem{rem}[thm]{Remark}
\newtheorem{lem}[thm]{Lemma}
\newtheorem{prop}[thm]{Proposition}

\newcommand\R{{\mathbb{R}}}

\newcommand\N{\mathbb{N}}

\newcommand\eps{\varepsilon}

\newcommand{\e}{\varepsilon}

\newcommand{\widebar}[1]{\mbox{\kern1.5pt\hbox{\vbox{\hrule height 0.6pt 
\kern0.35ex \hbox{\kern-0.15em \ensuremath{#1 }\kern0.0em}}}}\kern-0.1pt}

\newcommand{\E}{\mathbb{E}}
\renewcommand{\P}{\mathbb{P}}

\newcommand{\abs}[1]{\left\vert #1 \right\vert}
\newcommand{\norm}[1]{\left\Vert #1 \right\Vert}

\newlength{\fixboxwidth}
\title{Tractability of the approximation of high-dimensional rank one tensors}

\author{Erich Novak,
Daniel Rudolf
\\
Mathematisches Institut, Universit\"at Jena\\
Ernst-Abbe-Platz 2, 07743 Jena, Germany\\
email: 
erich.novak@uni-jena.de, 
daniel.rudolf@uni-jena.de}

\begin{document}

\maketitle

\begin{abstract}
We study the approximation of high-dimensional rank one tensors 
using point evaluations and consider deterministic as well as 
randomized algorithms. 
We prove that for certain parameters 
(smoothness and norm of the $r$th derivative) 
this problem is intractable while for other parameters 
the problem is tractable and the complexity is only
polynomial in the dimension for every fixed $\e >0$. 
For randomized algorithms 
we completely characterize the set of parameters 
that lead to easy 
or difficult problems, respectively. 
In the ``difficult'' case we modify the class to obtain 
a tractable problem: 
The problem gets tractable with a polynomial 
(in the dimension) complexity  if the support 
of the function is not too small. 
\end{abstract}


{\bf Classification. Primary: 65Y20; Secondary: 41A25, 41A63, 65C05.}

\section{Introduction}  \label{sec1} 

Many real world problems are high-dimensional, they 
involve functions $f$ that depend on many variables. 
It is known that the approximation of functions 
from certain smoothness 
classes suffers from the curse of dimensionality,
i.e., the complexity (the cost of an optimal algorithm) 
is exponential in the dimension $d$. 
The recent papers \cite{HNUW12,NW09} contain such results 
for classical $C^k$ and also $C^\infty$ spaces, 
the known theory is presented 
in the books \cite{NW08,NW10,NW12}. 
To avoid this curse of dimensionality one studies problems with 
a structure, 
see again the monographs just mentioned and~\cite{MUV14}. 

One possibility is to assume that the function $f$, 
say $f: [0,1]^d \to \R$, is a tensor of rank one, i.e., $f$ 
is of the form 
\begin{equation*} 
f(x_1, x_2, \dots , x_d) = \prod_{i=1}^d f_i(x_i), 
\end{equation*} 
for $f_i\colon [0,1] \to \R$. 
For short we also write $f = \bigotimes_{i=1}^d f_i$.
At first glance, the ``complicated'' function $f$ 
is given by $d$ ``simple'' functions.
One might hope that with this model assumption the curse of 
dimensionality can be avoided. 

In the recent paper \cite{BDDG} the authors investigate how well a rank one 
function can be captured (approximated in $L_\infty$) 
from $n$ 
point evaluations. 
They use the function classes
\begin{equation*} 
F^r_{M,d} = \big\{ f \mid f = \bigotimes_{i=1}^d  f_i, \
\Vert f_i \Vert_\infty \le 1, \
\Vert f_i^{(r)} \Vert_\infty \le M \big\}
\end{equation*} 
and define an algorithm $A_n$ that uses $n$ function 
values.
Here, given an integer $r$ it is assumed that $f_i \in W_\infty^r[0,1]$, 
where $W_\infty^r[0,1]$ is the set of
all univariate functions on $[0,1]$ which have 
$r$ weak derivatives in $L_\infty$, and
$f_i^{(r)}$ is the $r$th weak derivative.

In \cite{BDDG} the authors consider an algorithm which consists of two phases. 
For $f\in F^r_{M,d}$, the first phase is looking for 
a $z^* \in [0,1]^d$ such that $f(z^*) \not = 0$, the
second phase takes this $z^*$ and constructs an approximation of $f$. 
The main error bound 
\cite[Theorem~5.1]{BDDG}
distinguishes two cases:
\begin{itemize}
\item 
If in the first phase no $z^*$ with $f(z^*) \not = 0$ 
was found, then $f$ itself
is close to zero, i.e. $A_n(f)=0$ satisfies
\begin{equation}  \label{bound} 
\Vert f - A_n(f) \Vert_\infty \le 
C_{d,r} M^d n^{-r} . 
\end{equation} 
\item 
If such a point $z^*$ is given in advance, 
then the second phase returns an approximation $A_n(f)$ 
and the bound 
\begin{equation} \label{z_known}
\Vert f - A_n(f) \Vert_\infty \le C_r \,M d^{r+1} n^{-r}
\end{equation}
holds.
Here $C_r>0$ is independent of $d,M,n$ and $n \ge d \max\{   (d C_r M )^{1/r},2\}$.

\end{itemize}

\begin{rem} 
The error bounds of \eqref{bound} and \eqref{z_known} 
are nice since the order of convergence
$n^{-r}$ is optimal. 
In this sense, which is the traditional point of view in numerical 
analysis, the authors of \cite{BDDG} correctly call 
their algorithm an optimal algorithm.
When we study the tractability of a problem 
we pose a different problem and want to know
whether 
the number of function evaluations,
for given spaces and an error bound 
$\e >0$, increases exponentially in the dimension $d$ or not. 
The curse of dimensionality may happen even for $C^\infty$ 
functions where the order of convergence is excellent, 
see \cite{NW09}.  

Consider again the bound \eqref{bound}. 
This bound is proved in \cite{BDDG} 
with the Halton sequence 
and hence 
we only have a non-trivial 
error bound if 
\begin{equation*} 
n  > 
\bigg( 2^d \prod_{i=1}^d p_i \bigg) \,  (2M)^{d/r} 
\end{equation*} 
where $p_1, p_2, \dots , p_d$ 
are the first $d$ primes. 
The number $n$ of needed function values for given parameters 
$(r, M)$ and error bound $\e <1$ increases \emph{always},
i.e., for all $(r, M, \e)$,   (super-) exponentially 
with the dimension $d$.    \qed 
\end{rem} 

We ask whether this \emph{curse of dimensionality} 
is inherent in the problem or whether it can be avoided 
by a different algorithm. 
We give a complete answer in the case of randomized algorithms.
It depends on $r$ and $M$, but not on $\e$. 
The curse of dimensionality is present for 
the classes $F^r_{M,d}$ if and only if $M \ge 2^r r! $. 
For smaller $M$ we construct a randomized algorithm 
that, for any fixed $\e >0$, has polynomial (in $d$) cost. 

To precisely formulate the results we need some further notation.
We want to recover a function $f$ 
from a class $F_d$ of functions defined on $[0,1]^d$.  
We consider the worst case error of an algorithm $A_n$ 
on $F_d$ and stress that $F_d$, in this paper, 
is not the unit ball with respect to some norm
since it is not convex. 
Hence we can not apply results that are based 
on this assumption, 
in particular 
we allow (and should allow) all adaptive 
algorithms
\begin{equation*}    \label{algo}    
A_n(f) = \phi ( f(x_1), f(x_2), \dots , f(x_n)) , 
\end{equation*}    
with $\phi \colon \R^n \to L_\infty$,
where the $x_i \in [0,1]^d$ can be chosen adaptively, 
depending  on the already known function values 
$f(x_1), \dots , f(x_{i-1})$. 
See, for example, 
\cite{No96,NW08}. 
The worst case error of a deterministic algorithm
$A_n$ is defined as   
$$    
e^{\rm det}(A_n,F_d)=\sup_{f\in F_d}  \Vert f - A_n(f) \Vert_\infty ,
$$    
whereas the $n$th minimal worst case error is     
\begin{equation}  \label{eq:error}   
e^{\rm det}(n, F_d) = \inf_{A_n}  e^{\rm det}(A_n,F_d), 
\end{equation}     
where $A_n$ runs through the set of all deterministic
algorithms that use at most $n$ function values. 

\section{Curse of dimensionality for large values of $M$}  \label{sec1a} 

One might guess that the whole problem is easy, since $f$ is given by 
the $d$ univariate functions $f_1, f_2, \dots , f_d$. We will see that this 
is \emph{not} the case 
for the classes $F^r_{M,d}$ if $M \ge 2^r r!$.
Let us start with considering the initial error $e^{\rm det} (0, F^r_{M,d})$.
We have $e^{\rm det} (0, F^r_{M,d})=1$, 
since all 
possible inputs satisfy $-1 \le f \le 1$ and it is obvious that the trivial 
algorithm $A_0(f)=0$ 
is optimal, if we do not 
compute any information on $f$.  
Further there is a function $g\in W_\infty^r[0,1]$ 
with $\Vert g \Vert_\infty =1$ and $\Vert g^{(r)}\Vert_\infty =2^r r! $ 
such that the support of $g$ is $[0,1/2]$ or $[1/2, 1]$. 
The $2^d$ tensor products of such functions show that the initial 
error $1$ of the problem cannot be reduced by less than $2^d$ function
values.  We obtain the following result. 

\begin{thm}  \label{thm: curse}
Let $r \in \N$ and $M \ge 2^r r!$. Then 
\begin{equation*}  \label{bound1} 
e^{\rm det}(n , F^r_{M,d} ) 
% = \inf_{A_n} \sup_{f\in F_d} \norm{f-A_n(f)}_{\infty} 
= 1 \qquad \hbox{for} \qquad n=1,2, \dots, 2^d -1. 
\end{equation*} 
\end{thm}   

\begin{proof}
Assume that $A_n$ is a deterministic (possibly adaptive) algorithm and 
$n\leq 2^d -1$. 
Since $f_0=0$ is in the space $F^r_{M,d}$ 
there are function values 
$f(x_1)= \dots = f(x_n)=0$ that are computed for 
the function $f=f_0$. 
Since $n\leq 2^d -1$ there is at least one orthant 
of $[0,1]^d$ which contains no sample point. 
Without loss of generality we assume that this 
orthant is $[0,1/2]^d$. The function 
$f^+ = \bigotimes_{i=1}^d f_i$ with
\[
% f_i(y_i) = 2^r \max\{ 0, (\frac{1}{2}-y_i)^r \}
f_i(t) = 2^r \max\{ 0, (\frac{1}{2}-t)^r \}, \qquad t\in [0,1]  
\]
is zero on $[0,1]^d \setminus [0,1/2]^d$, 
is an element of $F^r_{M,d}$ for $M \ge 2^r r! $ 
and  $f^+ (0)=1$. 
By construction we have 
$f^+(x_1) = \dots f^+(x_n)=0$ and, hence, 
$A_n(f^+) = A_n(-f^+) = A_n(f_0)$, since $A_n$ cannot 
distinguish those three inputs. From 
$\Vert f^+ - (-f^+) \Vert_\infty =2$ we conclude that 
$e^{\rm det}(A_n,F_{M,d}^r) \ge 1$ and hence
that 
\[
e^{\rm det}(n, F^r_{M,d}) \ge 1.
\]
The inequality $e^{\rm det}( n, F^r_{M,d}) \le 1$ is trivial since the zero 
algorithm has error 1. 
\end{proof}

In this paper we also analyze randomized algorithms $A_n$, i.e., 
the $x_i$ and also $\phi$ may be chosen randomly, see
Section 4.3.3 of \cite{NW08}. 
Then the output $A_n(f)$ is a random variable and 
the worst case error of such an algorithm on a class $F_d$ is defined by
\begin{equation*}  \label{err:ran} 
e^{\rm ran}(A_n,F_d) = \sup_{f \in F_d}  
\left( \E (\Vert f - A_n (f) \Vert_{\infty} )^2 \right)^{1/2}.
\end{equation*} 
Similarly to \eqref{eq:error}
the numbers 
$e^{\rm ran} (n, F_d)$ are again defined by the infimum 
over all $e^{\rm ran}(A_n,F_d)$ but now of course we allow randomized algorithms. 

Theorem~\ref{thm: curse} is for deterministic algorithms.
Already the authors of \cite{BDDG} suggest that randomized
algorithms might be useful for this problem. We will see
that this is true
if $M < 2^r r!$ but not for larger $M$. 
This follows from the results of Section 2.2.2 in \cite{No88}. 

\begin{thm}  \label{thm: curse2}
Let $r \in \N$ and $M \ge 2^r r!$. Then 
\begin{equation*} \label{bound2} 
e^{\rm ran} (n , F^r_{M,d} ) \ge \frac{1}{2} \sqrt{2} \qquad \hbox{for} 
\qquad n=1,2, \dots, 2^{d-1}  . 
\end{equation*} 
\end{thm}   

\begin{proof} 
As in the proof of Theorem~\ref{thm: curse} we can construct $2^d$
functions $h_1,\dots,h_{2^d}$ such that the $h_i$ have disjoint supports, 
$\Vert h_i \Vert_\infty =1$  
and $\pm h_i \in F^r_{M,d}$. 
Therefore the statement follows 
(with the technique of Bakhvalov), see 
Section 2.2.2 in \cite{No88}
for the details. 
\end{proof}

%E   What do these lower bounds \eqref{bound1} and \eqref{bound2} mean? 
We may say that this assumption, $f$ being a rank one tensor, 
is not a good assumption to avoid the curse of dimension,
at least if we study the approximation problem with 
standard information (function values) and if $M \ge 2^r r!$. 
If $M$ is ``large'' then the classes $F^r_{M,d}$ are too large 
and we have the curse of dimensionality. 
%E   What is the problem? 
A function $f \in F^r_{M,d}$ can be non-zero just in a small 
subset of the cube $[0,1]^d$ and still have a large norm 
$\Vert f \Vert_\infty$.
Then it might be difficult to find 
a $z^*$ with $f(z^*) \not= 0$. 
Only if we know such a point $z^*$ we can replace 
the multivariate problem by $d$ univariate problems.

Then one can apply the second phase of 
the algorithm of \cite{BDDG}. 
This second phase has been completely analyzed in \cite{BDDG}. 
For future reference, we state this result 
as a lemma.

\begin{lem}  \label{lemma1} 
Let $r \in \N$ and $M>0$. 
Consider all $f \in F^r_{M,d}$ with $f \not= 0$ and assume that
a $z^* \in [0,1]^d$ is known such that 
$f(z^*) \not = 0$. 
Then, if $n>d \max\{(d C_r M)^{1/r},2\}$,  there
is an algorithm $A_n$ with 
\begin{equation*} 
\norm{f-A_n(f)}_{\infty}\leq C_r \,M  d^{r+1}   n^{-r} . 
\end{equation*} 
\end{lem} 

The algorithm $A_n$ for Lemma~\ref{lemma1} 
from \cite{BDDG} is completely constructive; 
this is important since, in the present paper, we also 
speak about the \emph{existence} of algorithms in cases were 
we do \emph{not} have a construction.

We see two possibilities to obtain 
positive tractability results 
for the
approximation of high-dimensional rank one tensors
using function values. 
Both of them are considered in this paper:

\begin{itemize} 

\item
We study the same class $F^r_{M,d}$ but with 
``small'' values of $M$, i.e., 
$M < 2^r r! $. 
We do not have 
the curse of dimension for this class of functions
if we allow randomized algorithms, 
but of course we need other 
methods
than those of 
\cite{BDDG} to prove tractability.
See Section~\ref{sec2}. 

\item
We allow an arbitrary $M>0$ but study the smaller class 
\begin{equation*} 
F^{r,V}_{M,d} = \{ f \in F^r_{M,d}    \mid 
f(x) \not= 0 \hbox{ for all $x$ from a box with volume greater than $V$} \}.
\end{equation*} 
By a box we mean a set of the form 
$\prod_{i=1}^d [\alpha_i, \beta_i] \subset [0,1]^d$. 
If $V$ is not too small then, this is what we will prove, 
the problem is polynomially tractable
and the curse of dimensionality disappears. 
Again we need new algorithms to prove this result. 
In Section~\ref{sec3} we study deterministic as well as randomized 
algorithms. 

\end{itemize} 

We end this section with a few more definitions 
and remarks.  
Sometimes it is more convenient to discuss the inverse function
of $e^{\rm det}(n, F_d)$,
\begin{equation*} 
n^{\rm det}(\e, F_d) = \inf \{ n \mid e^{\rm det}(n, F_d) \le \e \} ,
\end{equation*} 
instead of $e^{\rm det}(n, F_d)$ itself.
The numbers 
$n^{\rm ran} (\e, F_d)$ are defined similarly. 
We say that a problem suffers from the curse of 
dimensionality in the deterministic setting, 
if $n^{\rm det}(\e , F_d) \ge C \, \alpha^d$ for some $\e >0$,
where $C>0$ and $\alpha >1$.

In this paper we say that
the complexity in the deterministic setting is polynomial in the dimension 
if
$n^{\rm det}(\e, F_d) \le C \, d^\alpha$ for each
fixed $\e >0$, where 
$C>0$ and $\alpha >1$ may depend on $\e$. 
We stress that the notions ``polynomially tractable'' 
and 
``quasi-polynomially tractable'' that are used in the literature 
are more demanding and are not used in Section~\ref{sec2}.
By replacing $n^{\rm det}(\e, F_d)$ by $n^{\rm ran}(\e, F_d)$
the curse of dimensionality and ``polynomial 
in the dimension''
are defined also in the randomized setting.

\section{Tractability for small values of $M$} \label{sec2} 

Here we 
study the class $F^r_{M,d}$ and assume that 
$M < 2^r r! $ and $\eps \in (0,1)$. 
We show that 
we do not have 
the curse of 
dimensionality
for this class of functions.

We start 
with a simple observation which follows by standard error bounds 
for polynomial interpolation of a smooth function.

\begin{lem} \label{lem: poly_interpol}
Let $a,b\in\R$ with $a<b$ and $g \in W^r_{\infty}[a,b]$, 
further assume that $g$ has $r$ distinct zeros.
Then
\begin{equation} \label{eq: virtue1}
\norm{g}_{\infty} \leq \norm{g^{(r)}}_{\infty}\frac{(b-a)^{r}}{r!} .
\end{equation}
If $\norm{g}_{\infty}\geq \eps$ and $\norm{g^{(r)}}_\infty \leq M$ 
we have
\begin{equation} \label{eq: virtue2}
\lambda_1(\{ g \not = 0 \}) \geq \left(\frac{r! \eps}{M}\right)^{1/r},
\end{equation}
where $\lambda_1$ denotes the $1$-dimensional Lebesgue measure.
\end{lem}

\begin{proof} 
If $p$ is the polynomial of degree less than $r$ that coincides
with $g$ at $r$ distinct points where $g$ is zero 
then $p=0$ 
% as well as 
and
$$
\Vert g-p \Vert_\infty \le 
\norm{g^{(r)}}_{\infty}\frac{(b-a)^{r}}{r!} .
$$
This proves \eqref{eq: virtue1}. 
Further, if $\norm{g}_\infty\geq \eps$ there 
is an interval $[a^*,b^*]\subseteq [a,b]$
with $a^*<b^*$ such that $\lambda_1(\{g\not = 0\}) 
\geq b^*-a^*$, $g\in  W^r_{\infty}[a^*,b^*]$ 
and $g(t)\geq \eps$ for some $t\in [a^*,b^*]$.
Thus, by \eqref{eq: virtue1} 
we have $\eps \leq (b^*-a^*)^r M/r! $ 
which implies \eqref{eq: virtue2}.
\end{proof}

Observe that if 
$f \in F^r_{M,d}$ and
at least one of the $f_i$ has at least $r$ distinct zeros,
then by Lemma~\ref{lem: poly_interpol}
$\Vert f \Vert_\infty \le \Vert f_i \Vert_\infty \le \frac{M}{r!}$
holds.
Assume now that $M$, $r$ and $\e$ are given with 
$M\leq r! \eps$.
Then there are only two cases: 
\begin{itemize}
\item
$\Vert f \Vert_\infty \le \frac{M}{r!} \leq \e$; 
in this case we can approximate $f$ by the zero function 
and this 
output is good enough, i.e., the error is bounded by $\e$. 
\item
All the sets $\{ x \in [0,1] \mid f_i(x)=0\}$ have 
less than $r$ elements and 
hence $\{ x \in [0,1]^d \mid f(x) = 0 \}$ has measure zero. 
\end{itemize} 

In the following we study randomized algorithms. For this
we denote by $(\Omega,\mathcal{F},\mathbb{P})$ the common probability space of all
considered random variables.

We consider the following randomized algorithm $S_{1,n}$.
\begin{algorithm}
Let $X$ be a uniformly distributed random variable in $[0,1]^d$.
Let $f\in F^r_{M,d}$ and $\omega \in \Omega$. Then $S_{1,n}$
works as follows:
 \begin{enumerate}
\item  
Set $x=X(\omega)$;\\
If $f(x) \not = 0$ then go to \ref{alg1_second_step};\\
Otherwise return $S_{1,n}(f)(\omega)=0$.

\item \label{alg1_second_step}
Run the algorithm of Lemma~\ref{lemma1} and return
$S_{1,n}(f)(\omega)=A_n(f)$.
% $A_n(f)$ of ,
\qed
\end{enumerate}
\end{algorithm}
This leads, by applying the error bound of Lemma~\ref{lemma1}, 
to the following result.

\begin{thm} \label{thm: simple_prob_bound}
Let $\eps>0$, $r\in \N$,   $M\in (0,r! \eps]$ 
and $n\geq d \max\{\eps^{-1/r} (d C_r M)^{1/r},2\}$. 
Then, for $f \in F^r_{M,d}$ we have
\[
\mathbb{P}(\norm{f-S_{1,n}(f)}_{\infty}\leq \eps) = 1.
\]
Hence 
\[
d \max\{\eps^{-1/r} (d C_r M)^{1/r},2\}+1 
\]
function values lead to an $\eps$ approximation with probability $1$.
\end{thm}

We give a numerical example: 
Let $r=5$ and $M=10$ and $\eps = 1/10$. 
Then the problem is easy, see Theorem~\ref{thm: simple_prob_bound}. 
A single function evaluation 
is enough (with probability $1$) for the first step 
of the algorithm. 
For $r=5$ and $M=120 \cdot 32$ and all 
$\eps <1$ the problem is difficult, see 
Theorem~\ref{thm: curse} and Theorem~\ref{thm: curse2}.  \qed 

Now we assume that $r$ and $\eps \in (0,1)$ are given and $M$ satisfies 
$
M\in (\eps r!, 2^r r!). 
$
We will construct a randomized algorithm with polynomial 
in $d$ cost. 
The idea is to search randomly an $x$ such that $f(x) \not= 0$. 
But a simple uniform random search in $[0,1]^d$ does not work efficiently. 
In particular, if $M$ is close to $2^r r!$, it may happen that 
the set $\{ f(x) \not= 0 \}$ is very small for an 
$f \in F^r_{M,d}$   with 
$\Vert f \Vert_\infty > \eps$. 
Thus, the probability to find a non-zero can be exponentially small 
with respect to the dimension, such that this simple uniform random 
search does not lead to a good
algorithm.

The observation of the next lemma is useful to obtain 
a more sophisticated search strategy.
For this we define 
\begin{align*}
\delta^* & = \left( \frac{1}{2^{r+1}}+  \frac{r!}{2M} \right)^{1/r}-1/2 
\quad \hbox{and}    \\
d^* 	  & = \left \lceil 
\frac{\log \eps^{-1}}{\log(\frac{M}{2^{r+1}\,r!}+\frac{1}{2})^{-1}} 
\right \rceil
\end{align*}
and assume that $d\geq d^*$. 
By $\lambda_d$ we denote the $d$-dimensional Lebesgue measure and 
for $J\subset \N$ we write
$\abs{J}$ for the cardinality of $J$.

\begin{lem}  \label{lem: cruc_observ}
Let $M< 2^r r!$ and $\eps <1$. 
If $f\in F^r_{M,d}$ 
with $\norm{f}_{\infty} \geq \eps$
then at least $d-d^*$ of the functions $f_i$ satisfy 
$$
\lambda_1 (\{f_i =0\} \cap [1/2 - \delta^*, 1/2 + \delta^* ]) = 0.
$$ 
\end{lem}

\begin{proof}
We prove the assertion by contraposition.
Assume that there is $J\subset \{1,\dots, d\}$ 
with $\abs{J}> d^*$ such that for all $i\in J$ 
the function $f_i$ has at least $r$ zeros. 
Then by Lemma~\ref{lem: poly_interpol}
\begin{equation*}   \label{E02} 
\Vert f_i \Vert_\infty \leq (1/2 + \delta^* )^r \frac{M}{r!}
\end{equation*} 
for all $i \in J$. 
Because of $M< 2^r r!$, the choice of $\delta^*$ and 
the choice of $d^*$ we have
$
\left[(1/2 + \delta^* )^r \frac{M}{r!}\right]^{d^*} < \eps,
$
which finally leads to $\norm{f}_\infty < \eps$.
\end{proof}
This 
motivates the next algorithm 
denoted by $S_{n_1,n_2}$.
\begin{algorithm} \label{alg: M_klein}
 Let 
 \[
  K_{d^*} = \left\{ J\subseteq \{1,\dots,d\} \mid \abs{J}=d^* \right\}
 \]
 be the set of the coordinate sets of cardinality $d^*$ and let
 $Y=(Y_i)_{1\leq i\leq n_1}$ be an i.i.d. 
 sequence of uniformly distributed random variables in $K_{d^*}$.
 Independent of $Y$ let $Z=(Z_i)_{1\leq j \leq dn_1}$ be an i.i.d. sequence
 of uniformly distributed random variables in $[0,1]$. Further, note  that 
 \[
  s(Z_i) = \frac{1}{2}+ \delta^*(2Z_i-1)
 \]
 has uniform distribution in $[1/2-\delta^*,1/2+\delta^*]$.
 Then for $f\in F_{M,d}^r$ and $\omega \in \Omega$ the algorithm 
 $S_{n_1,n_2}$ works as follows:
\begin{enumerate} 
\item
For $1\leq i \leq n_1$ do\\[0.5ex] 
\hspace*{1ex} Set $I=Y_i(\omega)$;\\
\hspace*{1ex} For $1\leq j \leq d$ do\\[0.5ex] 
% \begin{enumerate}
\hspace*{4ex} If $j\in I$ then set $x_j=Z_{j+d(i-1)}(\omega)$. 
	      Otherwise set $x_j=s(Z_{j+d(i-1)}(\omega))$;\\[0.5ex]
% choose $x_j \in [1/2-\delta^*,1/2+\delta^*]$
% uniformly distributed.
\hspace*{1ex} If $f(x_1,\dots,x_d) \not = 0$ then go to \ref{it: second};\\ 
\hspace*{1ex} If $i=n_1$ and we did not find
$f(x)\not=0$ then return $S_{n_1,n_2}(f)(\omega)=0$.
\item \label{it: second}
Run the algorithm of Lemma~\ref{lemma1} and 
return $S_{n_1,n_2}(f)(\omega)=A_{n_2}(f)$.\qed
\end{enumerate} 
\end{algorithm}
Roughly the algorithm
chooses uniformly a coordinate set $I$ of cardinality $d^*$. If $j\in I$
then $x_j \in[0,1]$ is uniformly distributed otherwise $x_j$ is chosen uniformly
distributed in $[1/2-\delta^*,1/2+\delta^*]$. 
Then we check whether $f(x_1,\dots,x_d) \not = 0$. 
If this is the case
we apply $A_{n_2}$ from  Lemma~\ref{lemma1}.

We obtain the following error bound for this algorithm. 

\begin{lem}    \label{lem: simple_prob_bound2}
Let $\eps>0$, $r\in \N$, $M\in (r!\eps ,2^r r!)$ and 
$n_2\geq d \max\{\eps^{-1/r} (d C_r M)^{1/r},2\}$. 
Further let
\[
\alpha_{r,\eps,M} = 1+\frac{2^{r+1} r! \log\eps^{-1}}{(2^r r! -M)}
\quad \mbox{and} \quad
C_{r,\eps,M} = \left(\frac{3^r M}{ r!\eps} 
\right)^{\frac{\alpha_{r,\eps,M}}{r}} . 
\]
Then,
for $f \in F^r_{M,d}$ holds
\[
\mathbb{P}(\norm{f-S_{n_1,n_2}(f)}_{\infty}\leq \eps)
\geq 
1 - \left(1-\frac{d^{-\alpha_{r,\eps,M}}}{C_{r,\eps,M}}\right)^{n_1} . 
\]
\end{lem}

\begin{proof}
We assume that $\norm{f}_{\infty}\geq \eps$, otherwise the zero output is fine.
Then, by \eqref{eq: virtue2}
we have for any
$f_i$ that $\lambda_1(\{f_i \not = 0\}) 
\geq \left(\frac{r!\eps }{M}\right)^{1/r}$.
Let us denote the probability that we found $f(x)\not =0 $ 
in a single iteration 
of the first step of the algorithm $S_{n_1,n_2}$ by $\theta$. 
Further note that for every $1\leq i \leq n_1$ there are 
$\abs{K_{d^*}}=\binom{d}{d^*}$
many choices of the $d^*$ different coordinates in $I$.
Thus, by $\binom{d}{d^*} 
% \leq \frac{d^{d^*}}{d^*!} 
\leq \left(\frac{3d}{d^*}\right)^{d^*}$,
Lemma~\ref{lem: cruc_observ} and the fact that 
$\lambda_1(\{f_i \not = 0\})\geq 
\left(\frac{r!\eps }{M}\right)^{1/r}$ for any $i\in \{1,\dots,d\}$ it follows
\begin{align*}
\theta \geq \frac{\left(\frac{r! \eps}{M} \right)^{d^*/r}}{\binom{d}{d^*}}
\geq \left[\left(\frac{r! \eps}{M} \right)^{1/r} \frac{d^*}{3d} \right]^{d^*}
.
\end{align*}
Further by $1-y < \log y^{-1}$ 
for $y\in(0,1)$ we obtain
$
1 \leq d^* 
% \leq 1+\frac{2^{r+1} r! \log\eps^{-1} }{(2^r r! - M)}
\leq \alpha_{r,\eps,M}
.
$
Now by the choice of $n_2$, Lemma~\ref{lemma1} and the previous consideration
it follows that
\begin{align*}
   \mathbb{P}(\norm{f-S_{n_1,n_2}(f)}_{\infty}\leq \eps) 
 = 1-(1-\theta)^{n_1} 
\geq 1 - \left[1-\left(\frac{1}{3d} \left(\frac{r! \eps}{M}\right)^{1/r} 
\right)^{
% 1+\frac{2^{r+1} r! \log\eps^{-1} }{(2^r r! - M)}
\alpha_{r,\eps,M}
}
\right]^{n_1}.
\end{align*}
\end{proof}

\begin{thm} \label{thm: simple_prob_bound2}
Let $\eps>0$, $r\in \N$, $M\in (r!\eps ,2^r r!)$
and $0<p<1$. 
Then, with Algorithm~\ref{alg: M_klein} denoted by $S_{n_1,n_2}$, the constants
$C_{r,\eps,M}$, $\alpha_{r,\eps,M}$ of Lemma~\ref{lem: simple_prob_bound2} and 
\[
C_{r,\eps,M} \cdot d^{\,\alpha_{r,\eps,M}} \log p^{-1} 
+ d \max\{\eps^{-1/r} (d C_r M)^{1/r},2\}
\]
function values we obtain 
for $f \in F^r_{M,d}$ 
an $\eps$ approximation with probability $1-p$.
\end{thm}
\begin{proof}
The result is an immediate consequence
of Lemma~\ref{lem: simple_prob_bound2}.
\end{proof}

Observe
that for any fixed $r\in \N$, $\eps\in (0,1)$ and $M\in (0,2^r r!)$ 
the number of function values which lead to an $\eps$ approximation 
is polynomial in the dimension.

\section{Tractability for a modified class of functions}   \label{sec3} 

For large $M$ we have seen that there is the curse of 
dimensionality for the classes $F^r_{M,d}$. 
For $f\in F^r_{M,d}$ with $f= \bigotimes_{i=1}^d f_i$
it can be difficult to find a point where the function 
is not zero even if $\Vert f \Vert_\infty$ is large.
If we assume that
every $f_i$ is not zero on an interval with Lebesgue measure $\alpha_i\in [0,1]$ 
we have
\[
 \lambda_d(\{ f\not= 0\}) \ge \prod_{i=1}^d \alpha_i.
\]
The lower bound from Theorem~\ref{thm: curse} and Theorem~\ref{thm: curse2} 
stems from the fact 
that $\alpha_i= 1/2$ is possible for each $i$ 
and we obtain $\prod_{i=1}^d \alpha_i = 2^{-d}$, 
i.e., this volume is exponentially small. 
We admit that it is possible that all $\alpha_i$ are small and then we 
obtain the curse of dimensionality as described. 
In other applications it might happen that only a few of the $\alpha_i$ are 
small and then we can avoid the curse. 

This motivates to study 
a class of functions $F^{r,V}_{M,d}$ 
with large support, 
we assume 
that the numbers $\alpha_i$ are sufficiently large.
We denote by 
\[
R=\{ \Pi_{i=1}^d [a_i,b_i] 
\subseteq [0,1]^d \mid a_i, b_i 
\in[0,1],\,a_i\leq b_i, \; i=1,\dots,d  \}
\]
the set of all boxes in $[0,1]^d$, here 
$\lambda_d(A)$ is 
the Lebesgue measure of $A\subset \R^d$.
Then, let
\begin{equation*} 
F^{r,V}_{M,d} = \{ f \in F^r_{M,d}    \mid 
\exists A \in R\; \mbox{with}\; \lambda_d(A)>V\; 
\mbox{and}\; f(x) \not= 0,\;\forall x\in A\}.
\end{equation*}

The basic strategy for the approximation of $f\in F^{r,V}_{M,d}$
is to find a $z^*\in [0,1]^d$ with $f(z^*)\not = 0$ 
and after that apply Lemma~\ref{lemma1}.

For finding $z^*$ the following definition is useful to measure the
quality of a point set.
Let
\[
{\rm disp}(x_1,\dots,x_n) = 
\sup_{A \in R,\; A\cap \{x_1,\dots,x_n\} = \emptyset } \lambda_d(A)
\]
be the dispersion of the set $\{x_1,\dots,x_n\}$. 
The dispersion of a set is 
the largest volume of a box which does not 
contain any point of the set.
By $n^{\rm disp}(V,d)$ we denote the smallest number of points 
needed to have at least one point 
in every box with volume $V$, i.e.
\[
n^{\rm disp}(V,d) = \inf\{ n\in\N 
\mid \exists\, x_1,\dots,x_n \in [0,1]^d\; 
\mbox{with} \; {\rm disp}(x_1,\dots,x_n) \leq V\}.
\]

The authors of \cite{BDDG} consider as a point set 
the Halton sequence and 
use the following result
of \cite{DJ09,RT96} proved with this sequence. 
% \begin{prop}
Let $p_1,\dots,p_d$ be the first $d$ prime numbers then
\begin{equation*} 
n^{\rm disp}(V,d) \le \frac{2^d \prod_{i=1}^d p_i}{V} .
\end{equation*} 
The nice thing
is the dependence on
$V^{-1}$ which is of course 
optimal, already for $d=1$. 
The involved constant is, however, super-exponential 
in the dimension,
even for a point set with $2^d \prod_{i=1}^d p_i$ elements one only 
obtains the trivial bound $1$ of the dispersion. 

The quantity $n^{\rm disp}(V,d)$ is 
well studied.
The following result is due to Blumer, Ehrenfeucht, 
Haussler and Warmuth, see \cite[Lemma~A2.4]{BlEhHaWa89}. 
For this note that the test set of boxes 
has Vapnik-Chervonenkis dimension $2d$. 
Recall that by $(\Omega,\mathcal{F},\P)$ we denote the common probability space
of all considered random variables.

\begin{prop}   \label{prop: disp_bound}  
Let $(X_i)_{1 \leq i \leq n}$ be an i.i.d. sequence of uniformly distributed 
random variables mapping in $[0,1]^d$.
Then for any $0<V<1$ and $n \in \N$ 
\[
\P( {\rm disp}(X_1,\dots,X_n ) \leq V ) 
\geq 1-\left({e}\, n/d\right)^{2d} 2^{-V n/2} . 
\]
Thus
\begin{equation*} 
n^{\rm disp}(V,d) \le 16 d V^{-1} \log_2 (13 V^{-1}).
\end{equation*} 
\end{prop} 
This shows that the number of function values needed
to find $z^*$ with $f(z^*)\not= 0$ depends only 
linearly on the dimension $d$.
Proposition~\ref{prop: disp_bound}
leads to the following theorem.

\begin{thm}   \label{theorem4} 
Let $r \in \N$, $M\in(0,\infty)$ and $\eps,V \in (0,1)$. 
Then
\begin{equation*} 
n^{\rm det}(\e, F^{r,V}_{M,d} ) \le 
16 \,d \, V^{-1} \log_2(13 V^{-1}) 
+ d \max\{\eps^{-1/r}(d C_r M)^{1/r},2\}   , 
% d\,[\, (3d M  \,C(r))^{1/r} \eps^{-1/r} + 16V^{-1} \log_2 (13 V^{-1}) +1\,], 
% C(r) d   (\e^{-1/r} M^{1/r} + 8V^{-1} \log (13 V^{-1})).
\end{equation*} 
where $C_r$ 
comes from Lemma~\ref{lemma1} and
does not depend on $d,V, M$ and $\eps$.
\end{thm} 
 \begin{proof} 
If we have a point set with dispersion smaller than $V$, 
we know that every box with Lebesgue measure at least $V$ contains a point.
By Proposition~\ref{prop: disp_bound} we know there is such a point set 
with cardinality at most $16 \,d \, V^{-1} \log_2(13 V^{-1}) $.
By computing $f(x)$ for each $x$ of the point set we find a non-zero, since
$f\in F^{r,V}_{M,d}$. Then, by Lemma~\ref{lemma1} we need 
$d \max\{ \eps^{-1/r}(C_r d M )^{1/r},2\}$ 
more function values for an $\eps$ approximation.
By adding the number of function values we obtain the assertion.
\end{proof}

Therefore 
the information complexity of the problem in the deterministic setting
is at most quadratic in the dimension, in particular, 
the problem is polynomially tractable in the worst case setting.

Theorem~\ref{theorem4} has a drawback 
since it is only a result on the \emph{existence} 
of a fast algorithm. 
It is based on 
Proposition~\ref{prop: disp_bound}, 
which tells us that a uniformly distributed random point set satisfies the
bound on $n^{\rm disp}(V,d)$ with high probability.
As far as we know, an explicit construction of such 
point sets is not known. 

Because of this, 
we also present a randomized algorithm
$S_{n_1,n_2}$ which consists of
two steps. 
Here $n_1\in \N$ indicates 
the number of function evaluations for the first step 
and $n_2\in\N$ the function evaluations
for the second one. 
\begin{algorithm} \label{alg: new_class_ran}
Let $(X_i)_{1\leq i\leq n_1}$ be an i.i.d. sequence of uniformly distributed
random variables in $[0,1]^d$. For $f\in F^{r,V}_{M,d}$
and $\omega\in \Omega$ the method $S_{n_1,n_2}$ works as follows:
\begin{enumerate}
\item 
For $1\leq i \leq n_1$ do\\[0.5ex]
\hspace*{1ex} Set $x_i=X_i(\omega)$;\\
\hspace*{1ex} If $f(x_i)\not = 0$, then set $z^*=x_i$ and go to \ref{it: V_second};\\
\hspace*{1ex} If $i=n_1$ and we did not find 
$z^*$ with $f(z^*)\not =0$ then
return $S_{n_1,n_2}(f)(\omega)=0$.
\item \label{it: V_second}
Run the algorithm of Lemma~\ref{lemma1} and 
return $S_{n_1,n_2}(f)(\omega)=A_{n_2}(f)$.\qed
\end{enumerate}
\end{algorithm}

This randomized algorithm has typical advantages and disadvantages
compared to deterministic algorithms: 

\begin{itemize} 
\item
The advantage is that the randomized algorithm is even faster.
For the first phase of the algorithm 
(search of an $z^*$ such that $f(z^*) \not= 0$) 
the number of roughly $d V^{-1} \log (V^{-1})$ function evaluations
is replaced by roughly $V^{-1}$. 
\item
The disadvantage is that this algorithm can output a wrong result, 
even if this probability can be made arbitrarily small. 
\end{itemize}

We have the following error bound. 

\begin{thm} \label{thm: simple_prob_bound3}
Let $n_2\in \N$ with $n_2 \geq d \max\{ \eps^{-1/r}(C_r d M )^{1/r},2\}$.
Then
\begin{align*}  \label{eq: forall_outside}
    \P(\norm{f-S_{n_1,n_2}(f)}_{\infty} \leq \eps) 
%     & = \P(T \leq n_1) = \sum_{i=1}^{n_1} \P(T=i) \\
    & 
%     \leq \sum_{i=1}^{n_1} (1-V)^{i-1} V 
\geq 1-(1-V)^{n_1} 
\end{align*}
for $f\in F^{r,V}_{M,d}$ and $n_1\in \N$.
\end{thm} 

\begin{proof}
Let $(X_i)_{1 \leq i \leq n_1}$ be an i.i.d. sequence 
of uniformly distributed 
random variables with values in $[0,1]^d$ and let 
\[
 T = \min\{ i\in \N \mid f(X_i) \not = 0  \}.
\]
Because of the choice of $n_2$
we have by Lemma~\ref{lemma1} that the error is smaller than $\eps$
if we found $z^*$ in the first step of Algorithm~\ref{alg: new_class_ran}
denoted by $S_{n_1,n_2}$.
Thus
\begin{align*}
\P(\norm{f-S_{n_1,n_2}(f)}_{\infty} \leq \eps) 
& \geq \P(T \leq n_1) 
% = \sum_{i=1}^{n_1} \P(T=i) 
% = \sum_{i=1}^{n_1} (1-V)^{i-1} V 
= 1-(1-V)^{n_1}. \qedhere
 \end{align*}
\end{proof} 
\begin{rem}
For the result of Theorem~\ref{thm: simple_prob_bound3} it is enough 
to assume $\lambda_d(\{f\not =0\})>V$ and $f\in F_{M,d}^r$. Thus, it is not necessary to have a box
with Lebesgue measure larger $V$ such that $f$ is not zero on this box.
\end{rem}

Actually we can use a single 
sequence of uniformly i.i.d 
random variables $(X_i)_{1\leq i \leq n_1}$  
for any function $f\in F^{r,V}_{M,d}$ 
and still the probability that a
point $z^*$ is found decreases exponentially fast for increasing $n_1$. 
More exactly,  for $n_1 \in \N$ we obtain 

\begin{prop}  
Let $n_2\in \N$ with $n_2 \geq d \max\{ \eps^{-1/r}(C_r d M )^{1/r},2\}$.
Then
\begin{align} \label{eq: forall_inside}
\P(e^{\text{det}}(S_{n_1,n_2},F^{r,V}_{M,d})\leq \eps) 
% = \P(\sup_{f\in F^{r,V}_{M,d}} \norm{f-S_{n_1,n_2}(f)}_{\infty} \leq \eps) 
%  & = \P( D(\{ X_1,\dots,X_{n_1} \}) \leq V ) \\
 & \geq 1-\left( e n_1/d\right)^{2d} 2^{-V n_1/2}.
\end{align}
\end{prop} 

\begin{proof} 
Again $(X_i)_{1\leq i\leq n_1}$ is an i.i.d. sequence of uniformly 
distributed random variables in $[0,1]^d$ and 
$ n_2 $ is chosen 
such that the error bound of Lemma~\ref{lemma1} is smaller than $\eps$
if we found $z^*$ in the first step.
Let us denote
$S_{n_1,n_2}(f,\omega)$ for $S_{n_1,n_2}(f)$ which uses
the points $x_i=X_i(\omega)$ for $1\leq i \leq n_1$.
Then
\begin{align*}
& \quad\; \{ \omega \in \Omega \mid {\rm disp}(X_1(\omega),\dots,X_n(\omega)) 
\leq V \} \\
& = \{ \omega \in \Omega \mid \sup_{A\in R,\; 
A\cap \{X_1(\omega),\dots,X_{n_1}(\omega) \}=\emptyset} \lambda_d(A) \leq V \} \\
& = \{ \omega \in \Omega \mid \forall A\in R \quad
\mbox{with} \quad \lambda_d(A) > V\quad \exists j\in\{1,\dots,n_1\}\quad 
\mbox{with} \quad X_j(\omega) \in A \}\\
& = \{ \omega \in \Omega \mid \forall f\in F^{r,V}_{M,d} \quad 
\exists \; j\in\{1,\dots,n_1\}\quad \mbox{with} 
\quad f(X_j(\omega)) \not = 0 \} \\
& \subseteq \{ \omega \in \Omega \mid \forall f\in F^{r,V}_{M,d} \quad
\mbox{holds}\quad \norm{f-S_{n_1,n_2}(f)(\omega)}_{\infty} \leq \eps \} \\
& = \{ \omega \in \Omega \mid \sup_{f\in F^{r,V}_{M,d}} 
\norm{f-S_{n_1,n_2}(f)(\omega)}_{\infty} \leq \eps \} .
\end{align*}
Finally, for $n_1\in \N$ we obtain by Proposition~\ref{prop: disp_bound}  
\begin{align*}
\P(\sup_{f\in F^{r,V}_{M,d}} \norm{f-S_{n_1,n_2}(f)}_{\infty} \leq \eps) 
\geq \P( {\rm disp}( X_1,\dots,X_{n_1} ) \leq V ) 
\geq 1-\left({e} n_1/d\right)^{2d} 2^{-V n_1/2}.
\end{align*}
\end{proof} 

By a simple argument 
we can also derive from \eqref{eq: forall_inside} the existence of a
``good''
deterministic algorithm.
Namely, for $n_1 \geq 16 d\, V^{-1} \log_2(13 V^{-1})$ 
the right-hand-side of \eqref{eq: forall_inside} is strictly larger than
zero, which implies that 
there exists a realization of $(X_i)_{1\leq i \leq n_1}$, 
say  $(x_i)_{1\leq i \leq n_1}$, such that
\[
\sup_{f\in F^{r,V}_{M,d}} \norm{f-S_{n_1,n_2}(f)}_{\infty} 
\leq \eps.
\]
This gives another proof of  Theorem~\ref{theorem4}.

\bigskip

\noindent
{\bf Acknowledgement.}
We thank Mario Ullrich, 
Henryk Wo\'zniakowski and two anonymous referees for valuable comments. 
This research was supported by the DFG-priority program 1324, 
the DFG Research Training Group 1523 and the ICERM at Brown University.

\end{document}